\documentclass[11pt]{article}
\usepackage{amssymb,amsmath,amsfonts,mathrsfs,amsthm,epsfig,latexsym,color}
\usepackage{enumitem,geometry}
\usepackage{fourier}
\usepackage[all]{xy}

\makeatletter
\newcommand{\subsectionruninhead}{\@startsection{subsection}{2}{0mm}
{-\baselineskip}{-0mm}{\bf\large}}
\newcommand{\subsubsectionruninhead}{\@startsection{subsubsection}{3}{0mm}
{-\baselineskip}{-0mm}{\bf\normalsize}}
\makeatother

\newtheorem*{theorem*}{Theorem}
\newtheorem*{proof*}{Proof}
\newtheorem*{proposition*}{Proposition}
\newtheorem*{notation*}{Notation}
\newtheorem*{corollary*}{Corollary}
\newtheorem*{claim*}{Claim}
\newtheorem*{remark*}{Remark}

\newtheorem{conjecture}{Conjecture}
\newtheorem{theorem}{Theorem}[section]
\newtheorem{proposition}{Proposition}[section]
\newtheorem{corollary}[proposition]{Corollary}
\newtheorem{lemma}[proposition]{Lemma}
\newtheorem{claim}[proposition]{Claim}

\theoremstyle{definition}

\theoremstyle{remark}
\newtheorem{remark}[proposition]{Remark}

\numberwithin{equation}{section}

\def\e{{\varepsilon}}

\begin{document}
\title{Rigidity of center Lyapunov exponents and $su$-integrability}
\author{Shaobo Gan  ~ and ~ Yi Shi}
\date{}
\maketitle
\begin{abstract}
Let $f$ be a conservative partially hyperbolic diffeomorphism, which is homotopic to an Anosov automorphism $A$ on $\mathbb{T}^3$. We show that the stable and unstable bundles of $f$ are jointly integrable if and only if every periodic point of $f$ admits the same center Lyapunov exponent with $A$. In particular, $f$ is Anosov. Thus every conservative partially hyperbolic diffeomorphism, which is homotopic to an Anosov automorphism on $\mathbb{T}^3$, is ergodic. This proves the Ergodic Conjecture proposed by Hertz-Hertz-Ures on $\mathbb{T}^3$.
\end{abstract}

\section{Introduction}\label{sec:introduction}

A diffeomorphism $f$ on a closed Riemannian manifold $M$ is partially hyperbolic if there exists a continuous $Df$-invariant splitting $TM=E^s\oplus E^c\oplus E^u$ and continuous functions $\sigma,\mu:M\rightarrow\mathbb{R}$, such that $0<\sigma<1<\mu$ and
$$
\|Df(v^s)\|<\sigma(p)<\|Df(v^c)\|<\mu(p)<\|Df(v^u)\|
$$
for every $p\in M$ and unit vector $v^*\in E^*(p)$, for $*=s,c,u$.

Since Pugh and Shub \cite{PS} conjectured that stably ergodic diffeomorphisms are open and dense in the space of $C^2$ conservative partially hyperbolic diffeomorphisms, ergodicity of partially hyperbolic diffeomorphisms has been one of the main topics of research in differentiable dynamics.
A key ingredient of proving ergodicity for partially hyperbolic diffeomorphisms is a property called accessibility. In dimension 3, for instance, it has been showed \cite{BW,HHU2} that every conservative accessible partially hyperbolic diffeomorphism is ergodic. Moreover, accessibility \cite{HHU2} is an open dense property for partially hyperbolic diffeomorphisms with one-dimensional center bundle. It seems promising that we can classify 3 dimensional non-ergodic partially hyperbolic diffeomorphisms. Actually, Hertz-Hertz-Ures proposed the following Ergodic Conjecture \cite{HHU1,CHHU}:

\begin{conjecture}\label{conj:ergodic-conjecture}
	If a conservative partially hyperbolic diffeomorphism of a 3-manifold is non-ergodic, then there is a 2-torus tangential to $E^s\oplus E^u$. In particular, the only orientable 3-manifolds that admit a non-ergodic conservative partially hyperbolic diffeomorphism are: 
	\begin{enumerate}
		\item the 3-torus $\mathbb{T}^3$;
		\item the mapping torus of ${\rm -Id}$; or
		\item the mapping torus of a hyperbolic automorphism of the 2-torus.
	\end{enumerate}
\end{conjecture}
 
The simplest 3-manifold supporting partially hyperbolic diffeomorphisms is 3-torus $\mathbb{T}^3$.
It has been proven in \cite{BI,Po2} that if $f:\mathbb{T}^3\rightarrow\mathbb{T}^3$ is partially hyperbolic, then the action $f_*:\pi_1(\mathbb{T}^3)=\mathbb{Z}^3\rightarrow\mathbb{Z}^3$ is also partially hyperbolic. This means $f_*\in{\rm GL}(3,\mathbb{Z})$ has three real eigenvalues with different modulos. One eigenvalue has modulo larger than $1$, and one has modulo smaller than one. So there are two classes of partially hyperbolic diffeomorphisms on $\mathbb{T}^3$:
\begin{itemize}
	\item either $f_*\in{\rm GL}(3,\mathbb{Z})$ has an eigenvalue equal to -1 or 1;
	\item or $f_*\in{\rm GL}(3,\mathbb{Z})$ is Anosov, i.e. every eigenvalue of $f_*$ has modulo not equal to 1.
\end{itemize}

In the first case, there are partially hyperbolic diffeomorphisms which are non-ergodic. For instance, an Anosov automorphism on 2-torus $\mathbb{T}^2$ times identity map on $\mathbb{S}^1$ is not ergodic. Moreover, it has been shown \cite{H2} that if such $f$ is not ergodic, then it admits 2-torus tangent to $E^s\oplus E^u$.

For the second case, it has been shown \cite{HU} that there is no 2-torus tangent to $E^s\oplus E^u$. Thus if we want to prove the Ergodic Conjecture on $\mathbb{T}^3$, we need to show that
every $C^2$ conservative partially hyperbolic diffeomorphism, homotopic to an Anosov automorphism on $\mathbb{T}^3$, is ergodic. See also \cite[Conjecture 1.11]{HU}.

In order to prove ergodicity for partially hyperbolic diffeomorphisms on 3-manifolds, the only obstruction is non-accessibility.
If $f$ is conservative, partially hyperbolic, and homotopic to an Anosov automorphism on $\mathbb{T}^3$, then $f$ is non-accessible implies that the stable and unstable bundles of $f$ are jointly integrable \cite{HU}. This is equivalent to $f$ admits a 2-dimensional invariant foliation tangent to the union of stable and unstable bundles everywhere. We say that such an $f$ is $su$-integrable. 

Hammerlindl and Ures proved the following theorem.

\begin{theorem*}[\cite{HU}]
Let $f$ be a $C^{1+\alpha}$ conservative partially hyperbolic diffeomorphism, which is homotopic to an Anosov automorphism $A$ on $\mathbb{T}^3$. If $f$ is not ergodic, it is topologically conjugate to $A$.
\end{theorem*}

Here $f$ is not ergodic is equivalent to $f$ is $su$-integrable and the integral $su$-foliation is minimal on $\mathbb{T}^3$. Moreover, Hammerlindl and Ures proved that the topological conjugacy preserves all invariant foliations of $f$, see Lemma \ref{lem:semi-conjugacy}.

In this paper, we give a necessary and sufficient condition for $su$-integrability of this kind of diffeomorphisms. Moreover, such kind of $f$ is Anosov by applying Lemma \ref{lem:Anosov}.

\begin{theorem}\label{thm:main}
	Let $f$ be a $C^{1+\alpha}$ conservative partially hyperbolic diffeomorphism, which is homotopic to an Anosov automorphism $A$ on $\mathbb{T}^3$. The stable and unstable bundles of $f$ are jointly integrable, if and only if, every periodic point of $f$ admits the same center Lyapunov exponent as $A$.
    In particular, $f$ is Anosov.
\end{theorem}

\begin{remark}
	In Theorem \ref{thm:main}, the condition that $f$ is conservative can be replaced by assuming the non-wandering set $\Omega(f)=\mathbb{T}^3$. Both properties imply that the $su$-foliation of $f$ is minimal and the conjugacy preserves the $su$-foliation.
\end{remark}

Combined with the work of Hammerlindl and Ures, we have the following corollary.

\begin{corollary}\label{coro:ergodic}
	Every $C^{1+\alpha}$ conservative partially hyperbolic diffeomorphism, which is homotopic to an Anosov automorphism on $\mathbb{T}^3$, is ergodic.
\end{corollary}

From the previous work of Ren, Gan and Zhang \cite{RGZ}, if $f$ is a $C^{1+\alpha}$ partially hyperbolic and Anosov diffeomorphism on $\mathbb{T}^3$, then there exist a series of equivalent conditions to $su$-integrability of $f$. We state them in Theorem \ref{thm:equivalent}.

\vskip 0.5 \baselineskip

\noindent {\bf Organization of this paper:} In Section \ref{sec:conjugacy}, we recall some properties of partially hyperbolic diffeomorphisms homotopic to an Anosov automorphism on $\mathbb{T}^3$. 
In Section \ref{sec:integrability}, we prove the ``sufficient'' part of Theorem \ref{thm:main}, which states the fact that all periodic points have the same center Lyapunov exponent implies $f$ is $su$-integrable.
In Section \ref{sec:exponents}, we show that if such kind of $f$ is $su$-integrable, then every periodic point of $f$ admits the same center Lyapunov exponent with $A$. This proves the ``necessary'' part of Theorem \ref{thm:main}. 
Finally, in Section \ref{sec:equivalent}, we give a series of equivalent conditions for $su$-integrability when $f$ is partially hyperbolic and Anosov on $\mathbb{T}^3$.

\vskip 0.5 \baselineskip

\noindent {\bf Acknowledgements:} We would like to acknowledge our debt to A. Gogolev  for a lot of help during preparing this paper, especially for pointing out that his work \cite{G} is useful for showing the rigidity of center Lyapunov exponents. We are grateful to A. Hammerlindl, F. Rodriguez Hertz, J. Rodriguez Hertz, A. Tahzibi, R. Ures, and J. Yang for their valuable comments.  S. Gan is supported by NSFC 11771025 and 11831001.  Y. Shi is supported by NSFC 11701015, 11831001 and Young Elite Scientists Sponsorship Program by CAST.

\section{Conjugacy and $su$-integrability}\label{sec:conjugacy}

Let $f$ be a partially hyperbolic diffeomorphism which is homotopic to an Anosov automorphism $A$ on $\mathbb{T}^3$. Then $A$ is also partially hyperbolic \cite{BI,Po2} $T\mathbb{T}^3=E^s_A\oplus E^c_A\oplus E^u_A$. These three invariant bundles are linear and corresponding to the three eigenvalues $\lambda_s,\lambda_c,\lambda_u$ of $A$ respectively. From now on, we assume that the center bundle of $A$ is expanding, i.e.
$$
|\lambda_s|<1<|\lambda_c|<|\lambda_u|.
$$
Denote by ${\cal F}^s_A,{\cal F}^c_A,{\cal F}^u_A$ the invariant foliations tangent to $E^s_A,E^c_A,E^u_A$ respectively. Since $A$ is linear, all bundles
$$
E^{cs}_A=E^s_A\oplus E^c_A, \qquad E^{cu}_A=E^c_A\oplus E^u_A, \qquad {\rm and} \qquad E^{su}_A=E^s_A\oplus E^u_A
$$
are integrable. Denote by ${\cal F}^{cs}_A,{\cal F}^{cu}_A,{\cal F}^{su}_A$ the foliations tangent to them respectively.

Since $f$ is partially hyperbolic, then $f$ has stable and unstable foliations ${\cal F}^s_f$ and ${\cal F}^u_f$ tangent to $E^s_f$ and $E^u_f$ respectively. It has been proved by R. Potrie that $f$ is dynamically coherent, i.e. there exist $f$-invariant foliations ${\cal F}^{cs}_f$ and ${\cal F}^{cu}_f$ tangent to $E^{cs}_f$ and $E^{cu}_f$ respectively. Moreover, ${\cal F}^{cs}_f$ intersects ${\cal F}^{cu}_f$ in an one-dimensional $f$-invariant foliation ${\cal F}^c_f$, which is tangent to $E^c_f$ everywhere. We denote by $d_{{\cal F}^*_f}(\cdot,\cdot)$ and $d_{{\cal F}^*_A}(\cdot,\cdot)$ be the distance induced by the inherited Riemannian metric on leaves of ${\cal F}^*_f$ and ${\cal F}^*_A$ , respectively, for $*=s,c,u,cs,cu$.

We denote by $\tilde{\cal F}^*_f$ and $\tilde{\cal F}^*_A$ the the lifting foliations of ${\cal F}^*_f$ and ${\cal F}^*_A$ in $\mathbb{R}^3$ for $*=s,c,u,cs,cu$. We denote by $d_{\tilde{\cal F}^*_f}(\cdot,\cdot)$ and $d_{\tilde{\cal F}^*_A}(\cdot,\cdot)$ be the distance induced by the inherited Riemannian metric on leaves of $\tilde{\cal F}^*_f$ and $\tilde{\cal F}^*_A$ , respectively, for $*=s,c,u,cs,cu$.

The following lemma was proved in \cite{Po2,HP}. See also \cite{BBI,H1} when $f$ is absolutely partially hyperbolic.

\begin{lemma}[\cite{Po2,HP}]\label{lem:quasi-isometric}
The two foliations $\tilde{\cal F}^s_f$ and $\tilde{\cal F}^{cu}_f$ have global product structure: $\tilde{\cal F}^s_f(x)$ intersects $\tilde{\cal F}^{cu}_f(y)$ in exactly one point, for every $x,y\in\mathbb{R}^3$. The two foliations $\tilde{\cal F}^u_f$ and $\tilde{\cal F}^{cs}_f$ have also global product structure.

The lifting foliation $\tilde{\cal F}^*_f$, $*=s,c,u$ is quasi-isometric in $\mathbb{R}^3$: there exist constants $a,b>0$, such that for any $y\in\tilde{\cal F}^*_f(x)$ with $*=s,c,u$,
$$
d_{\tilde{\cal F}^*_f}(x,y)\leq a\cdot|x-y|+b.
$$
\end{lemma}

\begin{lemma}[\cite{F,Po2,U,HU}]\label{lem:semi-conjugacy}
	Let $f$ be a $C^{1+\alpha}$ partially hyperbolic diffeomorphism which is homotopic to an Anosov automorphism $A$ on $\mathbb{T}^3$. There exists a continuous surjective map $h:\mathbb{T}^3\rightarrow\mathbb{T}^3$ satisfying:
	\begin{enumerate}
		\item $h\circ f=A\circ h$, taking a lift $H$ of $h$ and a lift $F$ of $f$, then $H\circ F=A\circ H$.
		\item $h$ is homotopic to identity, and there exists $L>0$, such that $\|H-{\rm Id}\|<L$.
		\item For every $\tilde{x}\in\mathbb{R}^3$, $H:\tilde{\cal F}^s_f(\tilde{x})\rightarrow\tilde{\cal F}^s_A(H(\tilde{x}))$ is a homeomorphism.
		\item For every $\tilde{x}\in\mathbb{R}^3$, $H(\tilde{\cal F}^*_f(\tilde{x}))=\tilde{\cal F}^*_A(H(\tilde{x}))$ for $*=c,cs,cu$.
		\item For every $x\in\mathbb{T}^3$, $h^{-1}(h(x))$ is a compact center arc with length at most $2aL+b$.
	\end{enumerate}
    If $f$ is $su$-integrable and $h$ is a homeomorphism, i.e. $f$ is topologically conjugate to $A$ by $h$, then $h$ preserves all invariant foliations
    $$
    h({\cal F}^*_f)={\cal F}^*_A, \qquad \forall *=c,s,u,cs,cu,su.
    $$
\end{lemma}

\begin{proof}
	Item 1 and 2 are well-known results by Franks \cite{F}. Item 3, 4 and 5 were proved by Potrie in \cite{Po2}. Item 5 see also \cite{U} for absolutely partially hyperbolic diffeomorphisms. The fact that $h$ is a conjugacy preserving all invariant foliations when $f$ is $su$-integrable was proved by Hammerlindl and Ures \cite{HU}. 
\end{proof}

In general, if $f$ is topologically conjugate to $A$ but not Anosov, then the conjugacy $h^{-1}$ is not H\"older continuous. However, we can show that $h^{-1}$ is H\"older continuous when restricted to every leaf of ${\cal F}^s_A$ and ${\cal F}^u_A$.

\begin{lemma}\label{lem:Holder}
	Under the assumption in Lemma \ref{lem:semi-conjugacy}, there exist constants $C>0$ and $0<\beta<1$, such that for every $x\in\mathbb{T}^3$ and $y\in{\cal F}^*_A(x)$, $*=s,u$, we have
	$$
	d_{{\cal F}^*_f}(h^{-1}(x),h^{-1}(y))\leq C\cdot d_{{\cal F}^*_A}(x,y)^{\beta}.
	$$
\end{lemma}

\begin{proof}
	We first prove this fact for $y\in{\cal F}^u_A(x)$. We fix $\e_0,\delta_0>0$, such that locally if $d_{{\cal F}^u_A}(x,y)<\delta_0$, then $d_{{\cal F}^u_f}(h^{-1}(x),h^{-1}(y))<\e_0$ for every $x\in\mathbb{T}^3$ and $y\in{\cal F}^u_A(x)$. Now we assume that
	$$
	d_{{\cal F}^u_A}(x,y)\ll\delta_0.
	$$
	Let $k$ be the largest positive integer such that $d_{{\cal F}^u_A}(A^kx,A^ky)<\delta_0$, then we have
	$$
	d_{{\cal F}^u_A}(x,y)>|\lambda_u|^{-(k+1)}\cdot\delta_0.
	$$
	
	On the other hand, we have
	$$
	d_{{\cal F}^u_f}(f^k\circ h^{-1}(x),f^k\circ h^{-1}(y))=d_{{\cal F}^u_f}(h^{-1}\circ A^k(x),h^{-1}\circ A^k(y))<\e_0.
	$$
	This implies
	$d_{{\cal F}^u_f}(h^{-1}(x),h^{-1}(y))<\mu^{-k}\cdot\e_0$, where $\mu=\inf_{z\in\mathbb{T}^3}m(Df|_{E^u_f(z)})>1$.
	
	If $\mu\geq|\lambda_u|$, then we have
	$$
	d_{{\cal F}^u_f}(h^{-1}(x),h^{-1}(y))<\frac{\delta_0}{|\lambda_u\cdot\e_0|}\cdot d_{{\cal F}^u_A}(x,y).
	$$
	Otherwise, we take $0<\beta<1$ such that $|\lambda_u|^{\beta}<\mu$. Then we have
	$$
	d_{{\cal F}^u_f}(h^{-1}(x),h^{-1}(y))<\mu^{-k}\cdot\e_0<|\lambda_u|^{-k\beta}\cdot\e_0< \frac{\e_0|\lambda_u|^{\beta}}{\delta_0}\cdot d_{{\cal F}^u_A}(x,y)^{\beta}.
	$$
	This proves that $h^{-1}$ is H\"older continuous on every leaf of ${\cal F}^u_A$. The proof for $y\in{\cal F}^s_A(x)$ is the same.
\end{proof}


\begin{notation*}
	Let $p\in{\rm Per}(f)$ be a periodic point of $f$ with period $\pi(p)$. We denote by
	$$
	\lambda_c(p)=\|Df^{\pi(p)}|_{E^c_f(p)}\|^{\frac{1}{\pi(p)}}.
	$$
	Then $\log\lambda_c(p)$ is equal to the center Lyapunov exponent of $p$. Moreover, we denote $\lambda_c(A)=|\lambda_c|>1$, and $\log\lambda_c(A)$ is equal to the center Lyapunov exponent of $A$.
\end{notation*}

\begin{lemma}\label{lem:growth-rate}
	Let $f$ be a $C^1$ partially hyperbolic diffeomorphism which is homotopic to an Anosov automorphism $A$ on $\mathbb{T}^3$. Then there exists a sequence of periodic points $\{p_n\}$ of $f$, such that $\lim_{n\rightarrow\infty}\lambda_c(p_n)\geq\lambda_c(A)$.
\end{lemma}

\begin{proof}
	From Lemma \ref{lem:semi-conjugacy}, let $F:\mathbb{R}^3\rightarrow\mathbb{R}^3$ be a lift of $f$ and $H:\mathbb{R}^3\rightarrow\mathbb{R}^3$ be a lift of the semi-conjugacy $h$.
	The map $H$ satisfies $|H(\tilde{x})-\tilde{x}|\leq L$ for every  $\tilde{x}\in\mathbb{R}^3$.
	We can choose two points $\tilde{x},\tilde{y}\in\mathbb{R}^3$, such that $\tilde{y}\in\tilde{{\cal F}}^c_f(\tilde{x})$ and $|\tilde{x}-\tilde{y}|=3L$. Then $|H(\tilde{x})-H(\tilde{y})|\ge L>0$ and $H(\tilde{y})\in\tilde{{\cal F}}^c_A(H(\tilde{x}))$.
	
	Denote by $J^c_f$ the the arc connecting $\tilde{x}$ and $\tilde{y}$ in $\tilde{{\cal F}}^c_f(\tilde{x})$, and $J^c_A$ the arc connecting $H(\tilde{x})$ and $H(\tilde{y})$ in $\tilde{{\cal F}}^c_A(H(\tilde{x}))$, then we have
	$$
	H(F^n(J^c_f))=A^n(J^c_A), \qquad \forall n\geq 0.
	$$
	
	Then for every $n$ large enough, we have
	$$
	|F^n(J^c_f)|\geq|A^n(J^c_A)|-2L>\frac{|J^c_A|}{2}\cdot\lambda_c(A)^n.
	$$
    (for a smooth arc $J$, $|J|$ denotes the arc length of $J$.)
	This implies that for every $n$ large enough, there exists $\tilde{x}_n\in J^c_f$, such that for $x_n=\pi(\tilde{x}_n)$,
	$$
	\frac{1}{n}\sum_{i=0}^{n-1}\log\|Df|_{E^c(f^i(x_n))}\|=\frac{1}{n}\sum_{i=0}^{n-1}\log\|DF|_{E^c(F^i(\tilde{x}_n))}\|
	>\log\lambda_c(A)+\frac{\log|J^c_A|-\log2|J^c_f|}{n}.
	$$
	
	Taking an accumulation point $\mu_0$ of the sequence of measures $\{\sum_{i=0}^{n-1}\delta_{f^i(x_n)}/n\}$, we get that $\mu_0$ is an invariant probability measure of $f$ and
	$$
	\int\log\|Df|_{E^c(x)}\|{\rm d}\mu_0(x)\geq\log\lambda_c(A).
	$$
	By ergodic decomposition theorem, we can assume $\mu_0$ is ergodic. Since $\mu_0$ is a hyperbolic measure, by Liao's shadowing lemma (e.g., see \cite{Gan1, Gan2}, there exists a sequence of periodic points $\{p_n\}$ of $f$, such that $\lim_{n\rightarrow\infty}\lambda_c(p_n)\geq\lambda_c(A)$.
\end{proof}

\begin{theorem}[\cite{BGY}]\label{thm:BGY}
	Let $p$ be a hyperbolic periodic point of a diffeomorphism $f$ on a compact manifold. Assume that its homoclinic class $H(p)$ admits a (homogeneous) dominated splitting $T_{H(p)}M=E\oplus F$ with $E$ contracting and ${\rm dim}(E)={\rm ind}(p)$. If $f$ is uniformly $F$-expanding at the period on the set of periodic points $q$ homoclinically related to $p$, then $F$ is uniformly expanding on $H(p)$.	
\end{theorem}

\begin{lemma}\label{lem:Anosov}
	Let $f$ be a $C^1$ partially hyperbolic diffeomorphism which is homotopic to an Anosov automorphism $A$ on $\mathbb{T}^3$. If $\lambda_c(p)=\lambda_c(q)$ for every $p,q\in{\rm Per}(f)$, then $f$ is Anosov.
\end{lemma}

\begin{proof}
	From Lemma \ref{lem:growth-rate}, we know that $\lambda_c(p)\geq\lambda_c(A)>1$ for every $p\in{\rm Per}(f)$. From the semi-conjugacy $h:\mathbb{T}^3\rightarrow\mathbb{T}^3$ in Lemma \ref{lem:semi-conjugacy}, $h(p)$ is a periodic point of $A$ for every $p\in{\rm Per}(f)$. Moreover, we have $h^{-1}(h(p))=\{p\}$. Otherwise, $h^{-1}(h(p))$ is an $f$-periodic center arc, which must contain a periodic point of $f$ admitting non-positive center Lyapunov exponents.
	
	This implies that for every $p\in{\rm Per}(f)$, the unstable manifold $W^u_f(p)={\cal F}^{cu}_p(f)$ which is dense in $\mathbb{T}^3$ and tangent to $E^{cu}_f$ everywhere. On the other hand, since $h$ restricted to every stable leaf ${\cal F}^s_f(x)$ is a homeomorphism to ${\cal F}^s_A(h(x))$. If $h$ is injective at a point $p$, then $h$ is injective at every point of ${\cal F}^s_f(p)$. Let $H_f(p)=\overline{W^s_f(p)\pitchfork W^u_f(p)}$ be the homoclinic class of $p$ w.r.t. $f$. Then we have
	$$
	h(H_f(p))=h(\overline{W^s_f(p)\pitchfork W^u_f(p)})=\overline{h(W^s_f(p)\pitchfork W^u_f(p))}=\overline{W^s_A(h(p))\pitchfork W^u_A(h(p))}=H_A(h(p))=\mathbb{T}^3.
	$$
	
	Now we consider the partially hyperbolic splitting $T_{H_f(p)}\mathbb{T}^3=E^s_f\oplus E^{cu}_f$. Since $\lambda_c(p)\geq\lambda_c(A)>1$ for every $p\in{\rm Per}(f)$,  $f$ is uniformly $E^{cu}_f$-expanding at the period on all the periodic points in $H_f(p)$. Applying Theorem \ref{thm:BGY}, $E^{cu}_f$ is uniformly expanding and $H_f(p)$ is a hyperbolic set of $f$. Since $h$ is injective at every point of $W^s_f(p)$, $W^s_f(p)\subset H_f(p)$. If $H_f(p)\not=\mathbb{T}^3$, $H_f(p)$ would be a proper repeller, which is contradictory to the conservativity of $f$. This proves that $f$ is Anosov.
\end{proof}

\begin{corollary}\label{cor:equal}
	Let $f$ be a $C^1$ partially hyperbolic diffeomorphism which is homotopic to an Anosov automorphism $A$ on $\mathbb{T}^3$. If $\lambda_c(p)=\lambda_c(q)$ for every $p, q\in{\rm Per}(f)$, then $\lambda_c(p)=\lambda_c(A)$.
\end{corollary}

\begin{proof}
	We only have to show that  there exists a sequence of periodic points $\{q_n\}$ of $f$, such that $$\lim_{n\rightarrow\infty}\lambda_c(q_n)\leq\lambda_c(A).$$
	
	This proof goes similarly with Lemma \ref{lem:growth-rate}.
	In fact, since $\tilde{{\cal F}}^c_f$ is quasi-isometric, there exist constants $a, b>0$, such that for every $n$ large enough,
	$$
	|F^n(J^c_f)|\leq a\cdot|F^n(\tilde{x})-F^n(\tilde{y})|+b\leq a\cdot(|A^n(J^c_A)|+2L)+b<2a|J^c_A|\cdot\lambda_c(A)^n.
	$$
	(for the definition of notations, see the proof of Lemma \ref{lem:growth-rate}.)
	So there exists $\tilde{y}_n\in J^c_f$, such that for $y_n=\pi(\tilde{y}_n)$,
	$$
	\frac{1}{n}\sum_{i=0}^{n-1}\log\|Df|_{E^c(f^i(y_n))}\|=\frac{1}{n}\sum_{i=0}^{n-1}\log\|DF|_{E^c(F^i(\tilde{y}_n))}\|
	<\log\lambda_c(A)+\frac{\log2a|J^c_A|-\log|J^c_f|}{n}.
	$$
	
	Taking an accumulation point $\mu_1$ of the sequence of measures $\{\sum_{i=0}^{n-1}\delta_{f^i(y_n)}/n\}$, we have that $\mu_1$ is an invariant probability measure of $f$ and
	$$
	\int\log\|Df|_{E^c(x)}\|{\rm d}\mu_1(x)\leq\log\lambda_c(A).
	$$
	By ergodic decomposition theorem, we can assume $\mu_1$ is ergodic. Since $f$ is Anosov, there exists a sequence of periodic points $\{q_n\}$ of $f$, such that $\lim_{n\rightarrow\infty}\lambda_c(q_n)\leq\lambda_c(A)$.
\end{proof}

The following theorem was essentially proved in the classical paper by Pugh-Shub-Wilkinson \cite{PSW}. We will need it in Section \ref{sec:exponents}.

\begin{theorem}[\cite{PSW}]\label{thm:PSW}
	Suppose that $f:M\rightarrow M$ is a $C^{1+\alpha}$ partially hyperbolic diffeomorphism with one-dimensional center bundle. If $f$ is dynamically coherent, then the local unstable and local stable holonomy maps are uniformly $C^1$ when restricted to each center unstable and each center stable leaf, respectively.
\end{theorem}

\section{Joint $su$-integrability}\label{sec:integrability}

In this section, we prove that if $f$ is a $C^{1+\alpha}$ conservative partially hyperbolic diffeomorphism on $\mathbb{T}^3$ which is homotopic to an Anosov automotphism, and the center Lyapunov exponent of every periodic point of $f$ is equal to $\log\lambda_c(A)$, then $f$ is $su$-integrable.

Firstly, we need the following lemma.

\begin{lemma}\label{lem:center-metric}
	Let $f$ be a $C^{1+\alpha}$ partially hyperbolic diffeomorphism which is homotopic to an Anosov automorphism $A$ on $\mathbb{T}^3$. If $\lambda_c(p)=\lambda_c(A)$ for every periodic point $p\in{\rm Per}(f)$, then there exists a continuous metric $d^c(\cdot,\cdot)$ defined on every leaf of center foliation ${\cal F}^c_f$, such that
	\begin{itemize}
		\item There exists $K>1$, satisfying
		$1/K\cdot d_{{\cal F}^c_f}(x,y)<d^c(x,y)<K\cdot d_{{\cal F}^c_f}(x,y)$, for every $y\in{\cal F}^c_f(x)$;
		\item $d^c(f(x),f(y))=\lambda_c(A)\cdot d^c(x,y)$, for every $y\in{\cal F}^c_f(x)$;
		\item The stable and unstable holonomy maps between center leaves are isometries under $d^c(\cdot,\cdot)$ when restricted to each center stable and center unstable leaf, respectively. 
	\end{itemize}
\end{lemma}	

\begin{proof}
	From Lemma \ref{lem:Anosov} and Corollary \ref{cor:equal}, we know that $f$ is Anosov and $\lambda_c(p)=\lambda_c(A)$ for every $p\in{\rm Per}(f)$. Then Livshits Theorem implies that there exists a H\"older continuous function $\phi:\mathbb{T}^3\rightarrow\mathbb{R}$, such that
	$$
	\log\|Df|_{E^c_f(x)}\|=\phi(x)-\phi\circ f(x)+\log\lambda_c(A), \qquad \forall x\in\mathbb{T}^3.
	$$
	This implies that
	$$
	\lambda_c(A)\cdot\exp(\phi(x))=\|Df|_{E^c_f(x)}\|\cdot\exp(f\circ\phi(x)),\qquad \forall x\in\mathbb{T}^3.
	$$
	
	Now we can define a metric on every leaf of ${\cal F}^c_f$ as the following: for every $y\in{\cal F}^c_f(x)$, let $\gamma:[0,1]\rightarrow{\cal F}^c_f(x)$ be a $C^1$-parametrization with $\gamma(0)=x$ and $\gamma(1)=y$, then
	$$
	d^c(x,y):=\int_{0}^{1}\exp(\phi\circ\gamma(t))\cdot|\gamma'(t)|{\rm d}t.
	$$
	Since $\phi$ is bounded, there exists $K>1$, such that
	$$
	\frac{1}{K}\cdot d_{{\cal F}^c_f}(x,y)<d^c(x,y)<K\cdot d_{{\cal F}^c_f}(x,y), \qquad \forall y\in{\cal F}^c_f(x).
	$$
	Moreover, the cohomology equation implies $f$ is conformal on ${\cal F}^c_f$ under this metric:
	$$
	d^c(f(x),f(y))=\lambda_c(A)\cdot d^c(x,y), \qquad \forall y\in{\cal F}^c_f(x).
	$$
	
	From this conformal structure, we know that for every $x\in\mathbb{T}^3$ and $z\in{\cal F}^u_f(x)$, we denote $h^u_{x,z}:{\cal F}^c_f(x)\rightarrow{\cal F}^c_f(z)$ the holonomy map induced by the unstable foliation ${\cal F}^u_f$ in ${\cal F}^{cu}_f(x)$, then
	$$
	d^c(h^u_{x,z}(y_1),h^u_{x,z}(y_2))=d^c(y_1,y_2), \qquad \forall y_1,y_2\in{\cal F}^c_f(x).
	$$
	The same property holds for $z\in{\cal F}^s_f(x)$ and the holonomy map $h^s_{x,z}:{\cal F}^c_f(x)\rightarrow{\cal F}^c_f(z)$ induced by stable foliation ${\cal F}^s_f$ in ${\cal F}^{cs}_f(x)$.
\end{proof}

\begin{remark}\label{rk:center-metric}
	If the function $\phi$ is the solution of the cohomology equation
	$$
	\log\|Df|_{E^c_f}\|=\phi-\phi\circ f+\log\lambda_c(A),
	$$
	then $\phi+\kappa$ is also the solution for every $\kappa\in\mathbb{R}$. The corresponding center metric $d^c_1(\cdot,\cdot)$ defined by $\phi+\kappa$ also satisfies all the properties in Lemma \ref{lem:center-metric}. Actually, they satisfy
	$$
	d^c_1(x,y)=e^{\kappa}\cdot d^c(x,y),\qquad\forall y\in{\cal F}^c_f(x).
	$$
\end{remark}

\begin{proposition}\label{prop:integrability}
	Let $f$ be a $C^{1+\alpha}$ partially hyperbolic diffeomorphism which is homotopic to an Anosov automorphism $A$ on $\mathbb{T}^3$. If $\lambda_c(p)=\lambda_c(A)$ for every periodic point $p\in{\rm Per}(f)$, then the stable and unstable bundles of $f$ are jointly integrable.
\end{proposition}

\begin{proof}
	Since $\lambda_c(p)=\lambda_c(A)$ for every periodic point $p\in{\rm Per}(f)$,
	let $d^c(\cdot,\cdot)$ be the metric on ${\cal F}^c_f$ which is defined in Lemma \ref{lem:center-metric}.
	
	If $E^s_f$ and $E^u_f$ are not jointly integrable, then we have 4-legs local twisting, i.e. there exist $x_0\in\mathbb{T}^3$, $y_0\in{\cal F}^s_f(x_0)$ and $z_0\in{\cal F}^u_f(x_0)$ which are very close to $x_0$ in the stable and unstable leaves of $x_0$, such that locally there exist $w_1\in{\cal F}^u_f(y_0)$ and $w_2\in{\cal F}^s_f(z_0)$ satisfying
	$$
	w_1\neq w_2, \qquad {\rm and} \qquad w_2\in{\cal F}^c_f(w_1).
	$$
	We denote $d^c(w_1,w_2)=\kappa_0>0$.
	
	\begin{claim}\label{clm:s-family}
		There exists a family of arcs ${\cal I}^s=\{I^s(x): x\in\mathbb{T}^3\}$ satisfying
		\begin{itemize}
			\item $I^s(x)\subset{\cal F}^s_f(x)$ admits $x$ as the start-point and varies continuously with respect to $x$.
			\item $I^s(x_0)$ admits $y_0$ as the end-point, and $I^s(z_0)$ admits $w_2$ as the end-point.
			\item Every $x_2\in{\cal F}^{cu}_f(x_1)$ satisfies that $I^s(x_2)=h^{cu}_{x_1,x_2}(I^s(x_1))$.
			\item There exist constants $0<l_1<l_2$, such that $l_1\leq|I^s(x)|\leq l_2$ for every $x\in\mathbb{T}^3$.
		\end{itemize}
	\end{claim}
	
	\begin{proof}[Proof of the claim]

		Let $I^s(x_0)$ be the arc from $x_0$ to $y_0$ in ${\cal F}^s_f(x_0)$, and $I^s(z_0)$ be the arc from $z_0$ to $w_2$ in ${\cal F}^s_f(z_0)$, then we can see that
		$$
		I^s(z_0)=h^{cu}_{x_0,z_0}(I^s(x_0)),
		$$
		where $h^{cu}_{x_0,z_0}:{\cal F}^s_f(x_0)\rightarrow{\cal F}^s_f(z_0)$ is the local holonomy map induced by ${\cal F}^{cu}_f$. Then for every point $x\in{\cal F}^{cu}_f(x_0)$, we can define
		$$
		I^s(x)=h^{cu}_{x_0,x}(I^s(x_0))\subset{\cal F}^s_f(x).
		$$
		Since every leaf of ${\cal F}^{cu}_f$ is homeomorphic to $\mathbb{R}^2$, and the lifting foliations $\tilde{\cal F}^{cu}_f, \tilde{\cal F}^s_f$ admit a global product structure,
		this tells us that $I^s(x)$ is well-defined for every point $x\in{\cal F}^{cu}_f(x_0)$.
		Moreover, the topological conjugacy $h$ maps ${\cal F}^{cu}_f$ into the linear ${\cal F}^{cu}_A$ implies that we can extend this family of stable arcs to $\mathbb{T}^3$:
		$$
		{\cal I}^s=\{I^s(x): x\in\mathbb{T}^3\}.
		$$
		Finally, since ${\cal F}^s_A$ and ${\cal F}^{cu}_A$ are linear foliations, the uniform continuity of $h$ gives us the constants $0<l_1<l_2$ such that $l_1\leq|I^s(x)|\leq l_2$ for every $x\in\mathbb{T}^3$.
	\end{proof}
	
	Symmetrically, we have the following claim.
	
	\begin{claim}\label{clm:u-family}
		There exists a family of arcs ${\cal I}^u=\{I^u(x): x\in\mathbb{T}^3\}$ satisfying
		\begin{itemize}
			\item $I^u(x)\subset{\cal F}^u_f(x)$ admits $x$ as the start-point and varies continuously with respect to $x$.
			\item $I^u(x_0)$ admits $z_0$ as the end-point, and $I^u(z_0)$ admits $w_1$ as the end-point.
			\item Every $x_2\in{\cal F}^{cs}_f(x_1)$ satisfies that $I^u(x_2)=h^{cs}_{x_1,x_2}(I^u(x_1))$.
			\item There exist constants $0<l_3<l_4$, such that $l_3\leq|I^u(x)|\leq l_4$ for every $x\in\mathbb{T}^3$.
		\end{itemize}
	\end{claim}
	
	We fix the orientation of $I^s(x_0)$ from $x_0$ to $y_0$ to be positive and assume it coincides with the positive orientation of ${\cal F}^s_f$. Since ${\cal F}^{cu}_f(x_0)$ is dense and ${\cal F}^s_f$ is orientable,  the orientation can be continuously extended to ${\cal I}^s$. Symmetrically, we fix the orientation of ${\cal I}^u$ which is positive from $x_0$ to $z_0$ at $I^u(x_0)$, and assume it coincides with the positive orientation of ${\cal F}^u$. Moreover, we assume that the arc from $w_1$ to $w_2$ has the same orientation with ${\cal F}^c_f$.
	
	For every $x\in\mathbb{T}^3$, we define the $su$-path $J^{su}(x)$ to be the path that goes through $I^s(x)$ to the end-point $y$ of $I^s(x)$, then go through $I^u(y)$ to the end-point $w'$. We call $w'$ the end-point of $J^{su}(x)$.
	Symmetrically, we can define the $us$-path $J^{us}(x)$ by going through $I^u(x)$ to the end-point $z$, then go through $I^s(z)$ to the end-point $w''$. We call $w''$ the end-point of $J^{us}(x)$.

	\begin{claim}\label{clm:c-family}
		There exists a family of arcs ${\cal I}^c=\{I^c(x): x\in\mathbb{T}^3\}$ satisfying
		\begin{itemize}
			\item $I^c(x)\subset{\cal F}^c_f(x)$ admits $x$ as the start-point and varies continuously with respect to $x$;
			\item For every $x\in\mathbb{T}^3$, denote $w'$ to be the end-point of $J^{su}(x)$ and $w''$ to be the end-point of $J^{us}(x)$, then $w''$ is the end-point of the arc $I^c(w')$. In particular, $w_2$ is the end-point of $I^c(w_1)$.
			\item For every $w'\in\mathbb{T}^3$ with $\partial I^c(w')=\{w',w''\}$, it satisfies
			$d^c(w',w'')=d^c(w_1,w_2)=\kappa_0>0$, and $I^c(w')$ from $w'$ to $w''$ has the same orientation with ${\cal F}^c_f$.	    	
		\end{itemize}	
	\end{claim}
	
	\begin{proof}[Proof of the claim]
		The definition of ${\cal I}^c$ comes from the second item of the claim. From the continuity of ${\cal I}^s$ and ${\cal I}^u$, and their holonomy invariance by ${\cal F}^{cu}_f$ and ${\cal F}^{cs}_f$, ${\cal I}^c$ is well defined and varies continuously. We only need to check the last item.
		
		For every $x\in{\cal F}^c_f(x_0)$, we denote $w'$ and $w''$ be the other endpoints of $su$-path $J^{su}(x)$ and $us$-path $J^{us}(x)$ respectively. The holonomy invariance of ${\cal I}^s$ and ${\cal I}^u$ implies $w',w''\in{\cal F}^c_f(w_1)$. Moreover, we consider the composition of holonomy maps $h^s_{x_0,y_0}:{\cal F}^c_f(x_0)\rightarrow{\cal F}^c_f(y_0)$ and $h^u_{y_0,w_1}:{\cal F}^c_f(y_0)\rightarrow{\cal F}^c_f(w_1)$, it is defined as
		$$
		h^{su}_{J^{su}(x_0)}:=h^u_{y_0,w_1}\circ h^s_{x_0,y_0}:{\cal F}^c_f(x_0)\rightarrow{\cal F}^c_f(w_1),
		$$
		where $h^{su}_{J^{su}(x_0)}(x)=w'$.
		
		Similarly, we have the holonomy map
		$$
		h^{us}_{J^{us}(x_0)}:=h^u_{z_0,w_2}\circ h^s_{x_0,z_0}:{\cal F}^c_f(x_0)\rightarrow{\cal F}^c_f(w_2)={\cal F}^c_f(w_1),
		$$
		which is the composition of the holonomy maps $h^u_{x_0,z_0}:{\cal F}^c_f(x_0)\rightarrow{\cal F}^c_f(z_0)$ and $h^s_{z_0,w_2}:{\cal F}^c_f(z_0)\rightarrow{\cal F}^c_f(w_2)$ and satisfies $h^{us}_{J^{us}(x_0)}(x)=w''$.
		
		Since the holonomy maps of stable and unstable foliations between center leaves are isometries under the metric $d^c(\cdot,\cdot)$ when restricted in each center-stable and center-unstable leaves, both $h^{su}_{J^{su}(x_0)}$ and $h^{us}_{J^{us}(x_0)}$ are isometries between ${\cal F}^c_f(x_0)$ and ${\cal F}^c_f(w_1)$ under the metric $d^c(\cdot,\cdot)$. This implies
		$$
		d^c(w_1,w')=d^c(x_0,x)=d^c(w_2,w'')=\kappa_0.
		$$
		So we have $d^c(w_1,w_2)=d^c(w',w'')$, that is $I^c(x)$ has the same length under the metric $d^c(\cdot,\cdot)$ for every $x\in{\cal F}^c_f(x_0)$. From the density of ${\cal F}^c_f(x_0)$ and continuity of ${\cal I}^c$, we prove the claim.		
	\end{proof}
	
	Now we lift these three family of arcs ${\cal I}^s$, ${\cal I}^u$ and ${\cal I}^c$ to the universal cover $\mathbb{R}^3$. We use the same notation for convenience.
	
	Now we fix $x^0\in\mathbb{R}^3$ and denote $z^0$ be the end-point of $I^u(x^0)$. Define inductively
	\begin{itemize}
		\item  $x^{i+1}\in\tilde{{\cal F}}^s_f(x^0)$ to be the end-point of $I^s(x^i)$ for $i=0,1,\cdots,n-1$;
		\item  $z^{i+1}\in\tilde{{\cal F}}^s_f(z^0)$ to be the end-point of $I^s(z^i)$ for $i=0,1,\cdots,n-1$.
	\end{itemize}
	Then we consider the end-point $w^0$ of $I^u(x^n)$, we can see that that $w^0\in\tilde{{\cal F}}^c_f(z^n)$. Moreover, there exists a sequence of points $\{w^0,w^1,\cdots,w^n\}\subset\tilde{{\cal F}}^c_f(z^n)$, such that
	\begin{itemize}
		\item $w^{i+1}$ is the end-point of $I^c(w^i)$ for $i=0,1,\cdots,n-1$;
		\item $w^n=z^n$ and $d^c(w^0,z^n)=n\cdot\kappa_0$.
	\end{itemize}

	\begin{figure}[htbp]
		\centering
		\includegraphics[width=15cm]{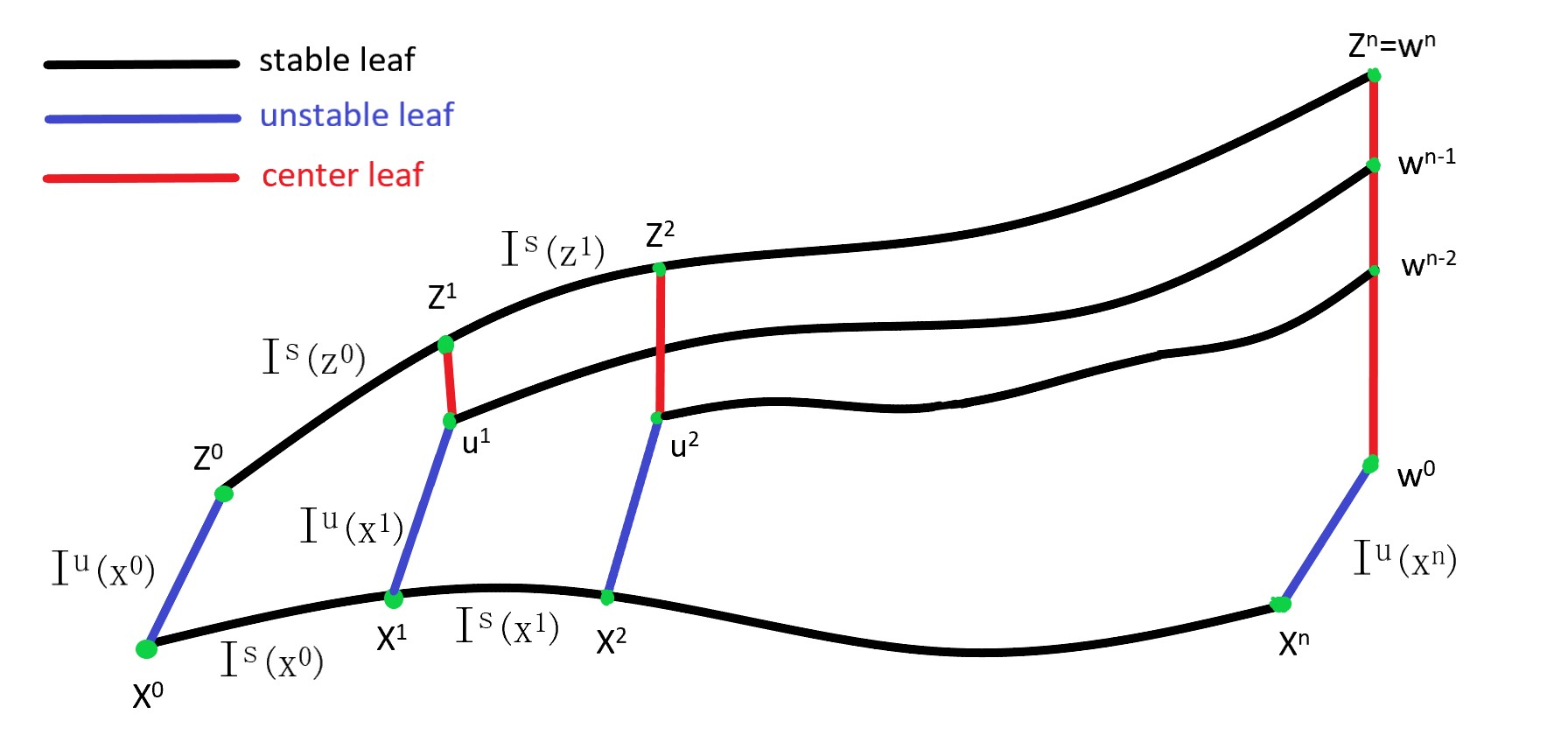}
		\caption{Global Twisting}
	\end{figure}
	
	Actually, if we denote $u^i$ to be the end-point of $I^u(x^i)$ for $i=1,\cdots,n-1$, we have
	$$
	w^i=\tilde{\cal F}^s_f(u^{n-i})\cap\tilde{\cal F}^c_f(w^0)\in\tilde{\cal F}^{cs}_f(w^0), \qquad i=1,\cdots,n-1.
	$$
	
	This implies $d^c(w^0,z^n)=n\cdot\kappa_0$. Since $\tilde{{\cal F}}^c_f$ is quasi-isometric, there exists $a>0$, such that $n\cdot a\kappa_0\leq |w^0-z^n|\rightarrow\infty$ as $n\rightarrow\infty$.
	Since $|x^n-w^0|\leq|I^u(x^n)|<l_4$, this implies
	$$
	|z^n-x^n|\longrightarrow\infty, \qquad {\rm as} \qquad n\rightarrow\infty.
	$$
	
	Let $F:\mathbb{R}^3\rightarrow\mathbb{R}^3$ be a lift of $f$, and $H:\mathbb{R}^3\rightarrow\mathbb{R}^3$ be the conjugacy satisfying $A\circ H=H\circ f$. Then there exists $L>0$, such that $|H(\tilde{x})-\tilde{x}|<L$.
	
	Since $H(\tilde{\cal F}^s_f)=\tilde{\cal F}^s_A$ and $H(\tilde{\cal F}^{cu}_f)=\tilde{\cal F}^{cu}_A$, we have
	$$
	H(z^0)\in\tilde{\cal F}^{cu}_A(x^0), \qquad
	H(x^n)\in\tilde{\cal F}^s_A(x^0), \qquad
	{\rm and} \qquad
	H(z^n)=\tilde{\cal F}^s_A(z^0)\pitchfork\tilde{\cal F}^{cu}_A(x^n).
	$$
	If we denote $\tilde{h}^s:\tilde{\cal F}^{cu}_A(x^0)\rightarrow\tilde{\cal F}^{cu}_A(x^n)$ as the holonomy map induced by the stable foliation $\tilde{\cal F}^s_A$, then we have
	$$
	H(x^n)=\tilde{h}^s(H(x^0)), \qquad {\rm and} \qquad
	H(z^n)=\tilde{h}^s(H(z^0)).
	$$
	
	However, since both $\tilde{\cal F}^s_A$ and $\tilde{\cal F}^{cu}_A$ are linear, we have
	$$
	|H(z^n)-H(x^n)|=|H(z^0)-H(x^0)|<|z^0-x^0|+2L\leq l_4+2L.
	$$
	This implies that for every $n$, we have
	$$
	|z^n-x^n|\leq|z^n-H(z^n)|+|x^n-H(x^n)|+(l_4+2L)<l_4+4L.
	$$
	This is a contradiction.	
\end{proof}

\section{Rigidity of center Lyapunov exponents}\label{sec:exponents}

In this section, we prove that if $f$ is a $C^{1+\alpha}$ conservative partially hyperbolic diffeomorphism on $\mathbb{T}^3$ which is homotopic to an Anosov automotphism and admits jointly integrable $su$-foliation ${\cal F}^{su}_f$, then the center Lyapunov exponent of every periodic orbit of $f$ is equal to $\log\lambda_c(A)$. 

From the work of Hammerlindl and Ures, the following proposition implies the ``necessary'' part of Theorem \ref{thm:main}. The idea of our proof originates from the work of A. Gogolev \cite{G}.

\begin{proposition}\label{prop:rigidity}
	Let $f$ be a $C^{1+\alpha}$ partially hyperbolic diffeomorphism which is homotopic to an Anosov automorphism $A$ on $\mathbb{T}^3$. If the stable and unstable bundles of $f$ are jointly integrable and $f$ is topologically conjugate to $A$, then
	$$
	\lambda_c(p)=\lambda_c(A), \qquad\forall p\in{\rm Per}(f).
	$$
	Thus $f$ is Anosov.	
\end{proposition}

\begin{proof}
	Recall we assumed that $\lambda_c(A)>1$. Since $f$ is topologically conjugate to $A$, the topological expanding in the center direction implies $\lambda_c(p)\geq 1$ for every periodic point $p$ of $f$.
	
	From Lemma \ref{lem:Anosov}, we only need to show that $\lambda_c(p)=\lambda_c(q)$ for any periodic points $p,q\in{\rm Per}(f)$.
	The topological conjugacy property implies that $f$ also satisfies the Shadowing Lemma. If there exist $p_1,p_2\in{\rm Per}(f)$, such that $\lambda_c(p_1)<\lambda_c(p_2)$, then the set
	$$
	\overline{\{\lambda_c(p):p\in{\rm Per}(f)\}}=[\lambda_-,\lambda_+]
	$$
	is a nontrivial interval contained in $[1,+\infty)$. By applying the Shadowing Lemma, we can take a smooth adapted Riemannian metric, such that
	$$
	\frac{\lambda_-}{1+\delta}<\|Df|_{E^c_f(x)}\|<\lambda_+\cdot(1+\delta), \qquad\forall x\in\mathbb{T}^3.
	$$
	Here $\delta$ could be arbitrarily small, and we will fix it later.
	
	Now we choose periodic points $p,q$ of $f$, such that
	$$
	\frac{\lambda_c(p)}{\lambda_-}\leq 1+\delta, \qquad {\rm and} \qquad \frac{\lambda_+}{\lambda_c(q)}\leq 1+\delta.
	$$
	Denote by $n_0$ the minimal common period of $p$ and $q$.
	
	Denote by $p'=h(p)$ and $q'=h(q)$ the conjugating periodic points of $A$. Then the strong unstable manifold $\tilde{\mathcal{F}}^u_A(p')$ is a line with irrational direction. This implies that an arc of ${\cal F}^u_A(p')$ with length $D'$ is $C_1/\sqrt{D'}$-dense in $\mathbb{T}^3$ for some $C_1>0$. From the local product structure, this implies that there exists $x'\in {\cal F}^u_A(p')$ and $y'\in{\cal F}^s_A(x')$, such that $q'\in{\cal F}^c_A(y')$. Moreover, there exists $C_2>0$ (independent of $D'$) such that
	$$
    d_{{\cal F}^u_A}(p',x')\leq D', \qquad d_{{\cal F}^s_A}(x',y')\leq\frac{C_2}{\sqrt{D'}}, \qquad {\rm and} \qquad
    d_{{\cal F}^c_A}(y',q')\leq\frac{C_2}{\sqrt{D'}}.
	$$
	
	Denote $x=h^{-1}(x')$ and $y=h^{-1}(y')$. We have $x\in{\cal F}^u_f(p)$, $y\in{\cal F}^s_f(x)$, and $q\in{\cal F}^c_f(y)$. From the continuity of the conjugacy $h$, for every $\eta>0$, there exists $D>0$, such that
	$$
	d_{{\cal F}^u_f}(p,x)\leq D, \qquad {\rm and} \qquad  d_{{\cal F}^c_f}(y,q)\leq\eta.
	$$
	By Lemma \ref{lem:Holder}, there exists $C_3>0$ and $0<\theta<1/2$, such that $d_{{\cal F}^s_f}(x,y)\leq C_3/D^{\theta}$. Recall that here constants $C_3$ and $\theta$ only depend on the contracting and expanding rates of $f$ on $E^s_f$ and $E^u_f$. Moreover, the points $x$ and $y$ also change here when $D$ changes. We will let $D$ tends to infinity in the future.
	
	Let $\eta_0>0$, such that for every $z_1,z_2\in\mathbb{T}^3$ satisfying $d(z_1,z_2)\leq3\eta_0$, we have
	$$
	\frac{\|Df|_{E^c(z_1)}\|}{(1+\delta)}<\|Df|_{E^c(z_2)}\|<(1+\delta)\cdot\|Df|_{E^c(z_1)}\|.
	$$
	From the fact that $f$ is conjugate to $A$ and from the uniform continuity of the conjugacy, there exists $0<\eta_1\leq\eta_0$, such that for any arc $J$ contained in a leaf of ${\cal F}^c_f$ with length $|J|\leq\eta_1$, it satisfies
	$$
	|f^{-n}(J)|\leq\eta_0, \qquad \forall n\geq0.
	$$
	
	Moreover, since the $su$-foliation ${\cal F}^{su}_A$ is linear, the uniform continuity of the conjugacy also shows that there exists $0<\eta_2\leq\eta_1$, such that for any arc $J$ contained in a leaf of ${\cal F}^c_f$ with length $|J|\leq\eta_2$, if $J'=h^{su}_f(J)$ is an arc contained in a leaf of ${\cal F}^c_f$ induced by the holonomy map $h^{su}$ of ${\cal F}^{su}_f$, it satisfies $|J'|\leq\eta_1$.
	
	Now we consider an arc $J_0\subset{\cal F}^c_f(p)$ with one endpoint $p$ and satisfying $|J_0|=\eta_2$, and we take $D$ large enough such that there exist $x\in{\cal F}^u_f(p)$ and $y\in{\cal F}^s_f(x)$ such that $q\in{\cal F}^c_f(y)$ and satisfy the following estimations:
	$$
	d_{{\cal F}^u_f}(p,x)\leq D, \qquad d_{{\cal F}^s_f}(x,y)\leq\frac{C_3}{D^{\theta}}\ll\eta_0, \qquad {\rm and} \qquad
	d_{{\cal F}^c_f}(y,q)\leq\eta_1.
	$$
	Let $J_1=h^{su}(J_0)$ admitting $x$ as one endpoint. This implies $|J_1|\leq\eta_1$. And we denote by $J^s(x,y)$ the arc contained in ${\cal F}^s_f(x)$ with endpoints $x$ and $y$; $J^c(y,q)$ the arc contained in ${\cal F}^c_f(y)$ with endpoints $y$ and $q$. Notice that when $D$ goes to infinity, all these estimations still hold.
	
    \begin{figure}[htbp]
		\centering
		\includegraphics[width=15cm]{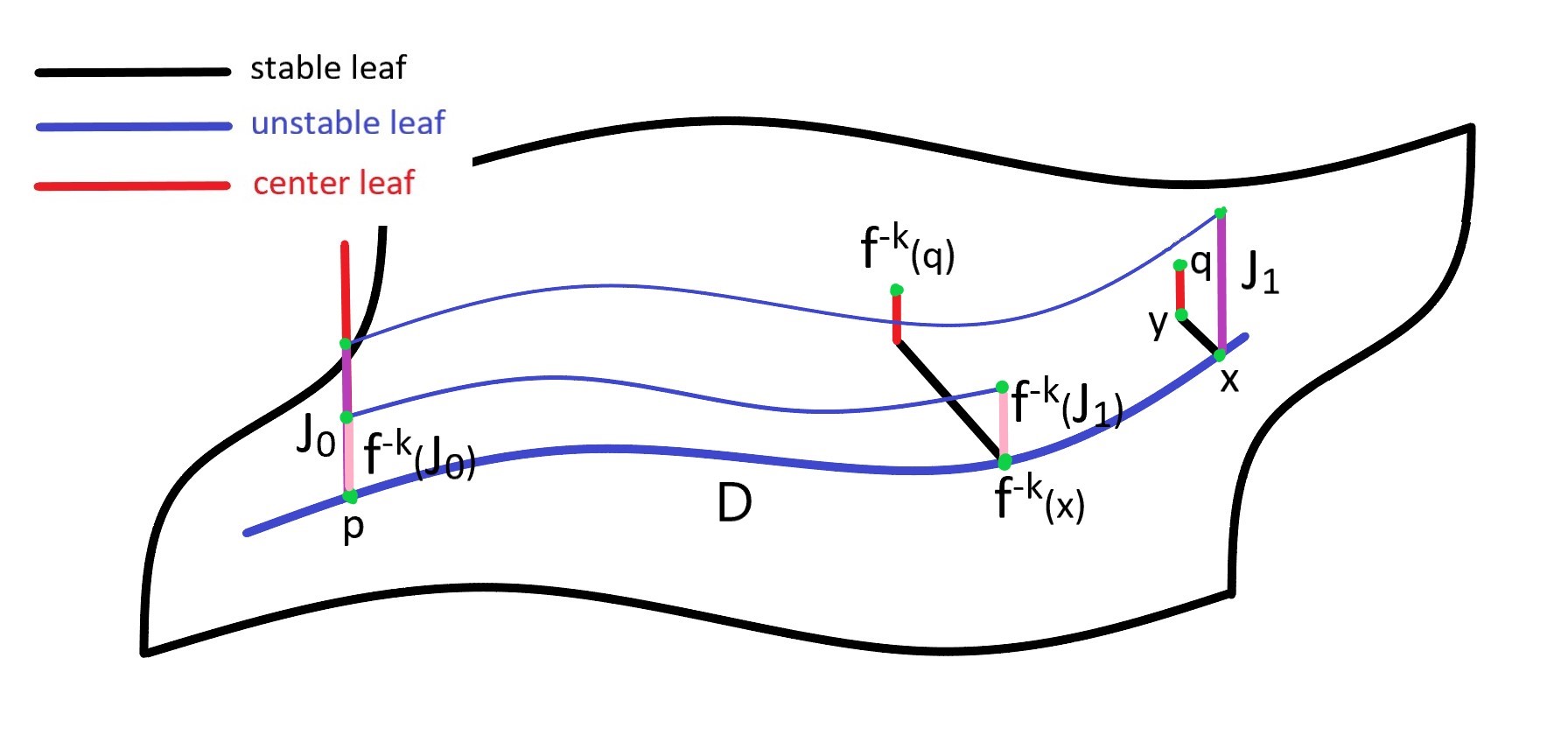}
		\caption{Holonomy Map}
	\end{figure}
	
	Denote by $N_0$ the first positive integer where $f^{-n_0N_0}(x)$ satisfies $d_{{\cal F}^u_f}(p,f^{-n_0N_0}(x))\leq 1$. Let $\mu=\sup_{x\in\mathbb{T}^3}\|Df^{-1}|_{E^u_f(x)}\|<1$, then
	$$
 N_0\leq\frac{\log D}{-n_0\log\mu}+1.
	$$
	And we have
	$$
	 |f^{-n_0N_0}(J_0)|\geq\frac{\lambda_c(p)^{-n_0N_0}}{(1+\delta)^{n_0N_0}}\cdot|J_0|\geq
	 \frac{\lambda_-^{-n_0N_0}}{(1+\delta)^{2n_0N_0}}\cdot|J_0|.
	$$
		
	On the other hand, denote $\gamma=\sup_{x\in\mathbb{T}^3}\|Df^{-1}|_{E^s_f(x)}\|>1$. We split $N_0=N_1+N_2$, such that $N_1$ is the largest integer satisfying
	$$
	|f^{-n_0N_1}(J^s(x,y))|\leq\eta_0.
	$$
	Since  $|J^s(x,y)|=d_{{\cal F}^s_f}(x,y)\leq C_3/D^{\theta}$, we have
	$$
	N_1\geq\frac{\theta\log D+\log\eta_0-\log C_3}{n_0\log\gamma}.
	$$
	Let
	$$
	\beta=\frac{1}{2}\cdot\frac{-\theta\log\mu}{\log\gamma},
	$$
	which only depends on the contracting and expanding rates of $f$ on stable and unstable bundles.
	
	Now we fix the constant $\delta$ so that it satisfies
	$$
	(1+\delta)^{[\frac{5}{\beta}]+1}\cdot\lambda_-<\lambda_+.
	$$
	Then we have
	\begin{align*}
	    \frac{N_1}{N_0}&\geq\frac{\theta\log D+\log\eta_0-\log C_3}{n_0\log\gamma}\cdot\frac{-n_0\log\mu}{\log D-n_0\log\mu} \\
	    &=\frac{-\theta\log\mu}{\log\gamma}\cdot\frac{1+\frac{\log\eta_0}{\theta\log D}-\frac{\log C_3}{\theta\log D}}{1-\frac{n_0\log\mu}{\log D}} \\
	    &=2\beta\cdot\frac{1+\frac{\log\eta_0}{\theta\log D}-\frac{\log C_3}{\theta\log D}}{1-\frac{n_0\log\mu}{\log D}}.
	\end{align*}
	So there exists $D_0>0$, such that if $D\geq D_0$, then we have
	$$
	N_1>\beta\cdot N_0.
	$$
	
	We can estimate the growth rate of $|J_1|$ now. For every $z\in J_1$, we have for every $0\leq k\leq n_0N_1$,
	$$
	d(f^{-k}(z),f^{-k}(q))\leq|f^{-k}(J_1)|+|f^{-k}(J^s(x,y))|+|f^{-k}(J^c(y,q))|\leq3\eta_0.
	$$
	This implies that
	$$
	|f^{-n_0N_1}(J_1)|\leq(1+\delta)^{n_0N_1}\lambda_c(q)^{-n_0N_1}|J_1|\leq(1+\delta)^{2n_0N_1}\lambda_+^{-n_0N_1}|J_1|.
	$$
	Since $\|Df|_{E^c_f(x)}\|>\lambda_-/(1+\delta)$ for every $x\in\mathbb{T}^3$, we have
	\begin{align*}
	  |f^{-n_0N_0}(J_1)|&\leq (1+\delta)^{n_0N_2}\lambda_-^{-n_0N_2}\cdot |f^{-n_0N_1}(J_1)|   \\
	             &<(1+\delta)^{n_0N_2}\lambda_-^{-n_0N_2}\cdot(1+\delta)^{2n_0N_1}\lambda_+^{-n_0N_1}|J_1|  \\
	             &<(1+\delta)^{2n_0N_0}\lambda_-^{-n_0N_2}\lambda_+^{-n_0N_1}|J_1|.
	\end{align*}
	
	Thus we have
	\begin{align*}
	\frac{|f^{-n_0N_0}(J_1)|}{|f^{-n_0N_0}(J_0)|}&< \frac{(1+\delta)^{2n_0N_0}\lambda_-^{-n_0N_2}\lambda_+^{-n_0N_1}}{\lambda_-^{-n_0N_0}\cdot(1+\delta)^{-2n_0N_0}}\cdot\frac{|J_1|}{|J_0|} \\
	&\leq (1+\delta)^{4n_0N_0}\cdot\left(\frac{\lambda_-}{\lambda_+}\right)^{\beta n_0N_0}\cdot\frac{|J_1|}{|J_0|} \\
	&<(1+\delta)^{-n_0N_0}\cdot\frac{\eta_1}{\eta_2}.
	\end{align*}
	
	When $D$ tends to infinity,  $N_0$ tends to infinity and $|f^{-n_0N_0}(J_1)|/|f^{-n_0N_0}(J_0)|$ tends to zero. Since $d_{{\cal F}^u_f}(p,f^{-n_0N_0}(x))\leq 1$, this implies that the holonomy map of unstable foliations restricted in ${\cal F}^{cu}_f(p)$ is not $C^1$-smooth. This contradicts to Theorem \ref{thm:PSW}, which states that these holonomy maps are locally uniformly $C^1$-smooth.
\end{proof}

\section{Equivalent conditions for $su$-integrability}\label{sec:equivalent}

From the proof of Proposition \ref{prop:integrability} and Proposition \ref{prop:rigidity},  we can see that if $f:\mathbb{T}^3\rightarrow\mathbb{T}^3$ is partially hyperbolic and Anosov, then $f$ is $su$-integrable as a partially hyperbolic diffeomorphism if and only if $\lambda_c(p)=\lambda_c(A)$ for every $p\in{\rm Per}(f)$.
Combined with the Main Theorem of \cite{RGZ}, we have a series of equivalent conditions to $su$-integrability of $f$.

\begin{theorem}\label{thm:equivalent}
	Let $f$ be a $C^{1+\alpha}$ partially hyperbolic and Anosov diffeomorphism, which is topologically conjugate to an Anosov automorphism $A$,  on $\mathbb{T}^3$. The following conditions are equivalent:
	\begin{enumerate}
		\item $f$ is $su$-integrable;
		\item $f$ is not accessible;
		\item The topological conjugacy $h$ ($h\circ f=A\circ h$) preserves unstable foliation of $f$: $h({\cal F}^u_f)={\cal F}^u_A$;
		\item The lifting unstable foliation $\tilde{\cal F}^u_f$ is homology bounded in $\mathbb{R}^3$, i.e. $\tilde{\cal F}^u_f(x)$ is uniformally bounded with $\tilde{\cal F}^u_A(x)$ for every $x\in\mathbb{R}^3$;
		\item $\lambda_c(p)=\lambda_c(A)$ for every periodic point $p\in{\rm Per}(f)$;
		\item The topological conjugacy $h$ is differentiable along ${\cal F}^c_f$.
	\end{enumerate}
\end{theorem}

\begin{proof}
	The equivalence from Item 1 to Item 4 has been proved in \cite[Main Theorem]{RGZ}. The equivalence between Item 1 to Item 5 has been proved Proposition \ref{prop:integrability} and Proposition \ref{prop:rigidity}. We only need to prove the equivalence between Item 5 to Item 6.
	
	\vskip 0.3 \baselineskip
	
	\noindent{\bf Item 5 $\Longrightarrow$ Item 6:} Let $p$ be a fixed point of $f$. The point $p'=h(p)$ is a fixed point of $A$. Now we choose a point $x\in{\cal F}^c_f(p)$, and denote by $J\subset{\cal F}^c_f(p)$ the center arc admitting $p,x$ as two endpoints. Then the points $p',x'=h(x)$ are endpoints of $J'=h(J)\subset{\cal F}^c_A(p')$.
	
	From Lemma \ref{lem:center-metric} and Remark \ref{rk:center-metric}, there exists a continuous metric $d^c(\cdot,\cdot)$ defined on every leaf of ${\cal F}^c_f$, satisfying all three properties in Lemma \ref{lem:center-metric} and
	$$
	d^c(p,x)=|J'|=d_{{\cal F}^c_A}(p',x').
	$$ 
	Here $|J'|$ is the length of arc $J'$. 
	
	\begin{claim}
		The conjugacy $h|_{J}:J\rightarrow J'$ is an isometry between $d^c(\cdot,\cdot)$ on $J$ and $d_{{\cal F}^c_A}(\cdot,\cdot)$ on $|\cdot|$ on $J'$.
	\end{claim}
	
	\begin{proof}[Proof of the claim]
		Denote by $x_{\frac12}\in J$ be the middle point between $p$ and $x$ under $d^c(\cdot,\cdot)$, i.e.
		$$
		d^c(p,x_{\frac12})=d^c(x_{\frac12},x).
		$$
		We want to show that $d_{{\cal F}^c_A}(p',h(x_{\frac12}))=d_{{\cal F}^c_A}(h(x_{\frac12}),x')$. 
		
		Since ${\cal F}^s_f(p)$ is dense in $\mathbb{T}^3$, there exists $y_n\in{\cal F}^s_f(p)$ such that $y_n\rightarrow x_{\frac12}$ as $n\rightarrow\infty$. Now we consider the holonomy map
		$$
		h^s_{p,y_n}:{\cal F}^c_f(p)\rightarrow{\cal F}^c_f(y_n).
		$$
		Since $h^s_{p,y_n}$ is an isometry under the metric $d^c(\cdot,\cdot)$ and $d^c(p,x_{\frac12})=d^c(x_{\frac12},x)$, we have
		$$
		h^s_{p,y_n}(x_{\frac12})\rightarrow x, \qquad{\rm as}~
		n\rightarrow\infty.
		$$
		
		On the other hand, $h({\cal F}^s_f(p))={\cal F}^s_A(p')$ implies $h(y_n)\in{\cal F}^s_A(p')$ and $h(y_n)\rightarrow h(x_{\frac12})$ as $n\rightarrow\infty$. Moreover,  we have
		$$
		h\circ h^s_{p,y_n}(x_{\frac12})\rightarrow x', \qquad{\rm as}~
		n\rightarrow\infty.
		$$
		This implies $d_{{\cal F}^c_A}(p',h(x_{\frac12}))=d_{{\cal F}^c_A}(h(x_{\frac12}),x')$.
		
		Repeating this procedure:
		\begin{itemize}
			\item denote by $x_{\frac14}$ the middle point between $p$ and $x_{\frac12}$ under $d^c(\cdot,\cdot)$, then we have
			$d_{{\cal F}^c_A}(p',h(x_{\frac14}))=d_{{\cal F}^c_A}(h(x_{\frac14}),h(x_{\frac12}))$;
			\item denote by $x_{\frac34}$ the middle point between $x_{\frac12}$ and $x'$ under $d^c(\cdot,\cdot)$, then we have
			$d_{{\cal F}^c_A}(h(x_{\frac12}),h(x_{\frac34}))=d_{{\cal F}^c_A}(h(x_{\frac34}),x')$.
		\end{itemize}
		
		Again, we take the middle points between $p$ and $x_{\frac14}$, $x_{\frac14}$ and $x_{\frac12}$, $x_{\frac12}$ and $x_{\frac34}$, $x_{\frac34}$ and $x'$, respectively. The same argument shows that $h$ preserves all the middle points between these intervals and their images by $h$. Repeating this procedure, form the density of these middle points, $h|_J:J\rightarrow J'$ is an isometry between $d^c(\cdot,\cdot)$ on $J$ and $d_{{\cal F}^c_A}(\cdot,\cdot)$ on $J'$. 	
	\end{proof}
	
	Recall that $d^c(f(x_1),f(x_2))=\lambda_c(A)\cdot d^c(x_1,x_2)$ for every $x_1,x_2\in J$ and $\|DA|_{E^c_A}\|\equiv\lambda_c(A)$. Since $h$
	is an isometry between $d^c(\cdot,\cdot)$ on $J$ and the natural distance on $J'$, it is an isometry between $d^c(\cdot,\cdot)$ on  ${\cal F}^c_f(p)$ and the natural distance on ${\cal F}^c_A(p')$.
	From the density of ${\cal F}^c_f(p)$ in $\mathbb{T}^3$, this showes that $h$ is an isometry between $d^c(\cdot,\cdot)$ on every leaf of ${\cal F}^c_f$ and the natural distance on every leaf of ${\cal F}^c_A$.
	
	Finally, for every $z\in{\cal F}^c_f(y)$ and $y\in\mathbb{T}^3$, let $\gamma:[0,1]\rightarrow{\cal F}^c_f(y)$ be a $C^1$-curve connecting $y$ and $z$, then
	$$
	d^c(y,z):=\int_{0}^{1}\exp(\phi\circ\gamma(t))\cdot|\gamma'(t)|{\rm d}t.
	$$
	Let $z\rightarrow y$, it implies 
	$$
	\|Dh|_{E^c_f(y)}\|=e^{\phi(y)}, \qquad\forall y\in\mathbb{T}^3,
	$$
	which proves that $h$ is differentiable in the center direction. 
	
	\vskip 0.3 \baselineskip
	
	\noindent{\bf Item 6 $\Longrightarrow$ Item 5:} Let $p\in{\rm Per}(f)$ be a periodic point of $f$ with period $\pi(p)$. Since $h$ is differentiable along ${\cal F}^c_f$, there exists a small arc $J\subset{\cal F}^c_f(p)$ containing $p$ and a constant $C>1$, such that for any subarc $I\subseteq J$, it satisfies
	$$
	\frac{1}{C}\leq\frac{|h(I)|}{|I|}\leq C.
	$$
	
	From the conjugacy, we have
	$$
	h\circ f^{-k\cdot\pi(p)}(J)=A^{-k\cdot\pi(p)}\circ h(J)\subseteq h(J), \qquad\forall k\geq0.
	$$
	Since $f$ is $C^{1+\alpha}$-smooth and both $f^{-1}$ is uniformly contracting in the center direction, the distortion control techniques showes that there exists another constant $K>1$, such that
	$$
	\frac{1}{K}\cdot\lambda_c(p)^{-k\cdot\pi(p)}<\frac{|f^{-k\cdot\pi(p)}(J)|}{|J|}<K\cdot\lambda_c(p)^{-k\cdot\pi(p)}, \qquad\forall k\geq0.
	$$
	On the other hand, we have $|A^{-k\cdot\pi(p)}(h(J))|=\lambda_c(A)^{-k\cdot\pi(p)}\cdot|h(J)|$ for every $k\geq0$.
	
	This showes that
	$$
	\frac{1}{K}\cdot\frac{\lambda_c(A)^{-k\cdot\pi(p)}\cdot|h(J)|}{\lambda_c(p)^{-k\cdot\pi(p)}\cdot|J|}<
	\frac{|A^{-k\cdot\pi(p)}(h(J))|}{|f^{-k\cdot\pi(p)}(J)|} <K\cdot\frac{\lambda_c(A)^{-k\cdot\pi(p)}\cdot|h(J)|}{\lambda_c(p)^{-k\cdot\pi(p)}\cdot|J|},
	\qquad\forall k\geq0.
	$$
	Since $1/C\leq|h(I)|/|I|\leq C$ for every $I\subseteq J$, we have
	$$
	\frac{1}{K\cdot C^2}<\frac{\lambda_c(A)^{-k\cdot\pi(p)}}{\lambda_c(p)^{-k\cdot\pi(p)}}<K\cdot C^2, \qquad\forall k\geq0.
	$$
	This proves $\lambda_c(p)=\lambda_c(A)$.
\end{proof}

\begin{remark}
	It should notice that we can build an $f$ such that its topological conjugacy is differentiable only in the center direction. Let $p\in\mathbb{T}^3$ be a fixed point of $A$. We composed a rotation around $p$ in the stable and unstable plane. For the new diffeomorphism, the stable and unstable Lyapunov exponents of $p$ are different from $A$. The topological conjugacy is differentiable in the center foliation.
	
	However, when $f$ is $C^1$-close to $A$, it has been showed by Gogolev and Guysinsky \cite{GG,G1} that the topological conjugacy is smooth if and only if all periodic points of $f$ admit the same three Lyapunov exponents as $A$. Thus the topological conjugacy is not differentiable.
\end{remark}

\noindent Shaobo Gan, School of Mathematical Sciences, Peking University, Beijing 100871, China\\
E-mail address: gansb@pku.edu.cn
\vspace{0.5cm}

\noindent Yi Shi, School of Mathematical Sciences, Peking University, Beijing 100871, China\\
E-mail address: shiyi@math.pku.edu.cn

\end{document}